\documentclass[11pt,american,preprint]{elsarticle}
\makeatletter
\def\ps@pprintTitle{%
 \let\@oddhead\@empty
 \let\@evenhead\@empty
 \def\@oddfoot{\centerline{\thepage}}%
 \let\@evenfoot\@oddfoot}
\makeatother

\usepackage[margin=1in]{geometry}

\usepackage{xcolor}
\usepackage{lipsum}
\usepackage{changepage}
\usepackage{centernot}

\usepackage[T1]{fontenc}
\usepackage[latin9]{luainputenc}
\usepackage{amsmath}
\usepackage{amsthm}
\usepackage{amssymb}
\usepackage{stmaryrd}
\usepackage{esint}
\usepackage{nameref}
\usepackage{hyperref} 
\usepackage{cleveref} %\autoref \nameref 

\usepackage[shortlabels]{enumitem}
\usepackage{chngcntr}

\usepackage{mathrsfs}

\newsavebox{\foobox}

\usepackage{graphicx,amssymb}

\usepackage{array}
\newcolumntype{M}[1]{>{\centering\arraybackslash}m{#1}}
\usepackage{bbm} %indicatorFunction
\makeatletter
\numberwithin{equation}{section}
\theoremstyle{plain}
\newtheorem{thm}{\protect\theoremname}[section]

\theoremstyle{plain*}
\newtheorem*{thm*}{\protect\theoremname}

\theoremstyle{plain}
\newtheorem{lem}[thm]{\protect\lemmaname}
  
\theoremstyle{plain*}
\newtheorem*{lem*}{\protect\lemmaname}  
  
  \theoremstyle{plain}

    \theoremstyle{plain*}
  \newtheorem*{prop*}{\protect\propositionname}

\theoremstyle{remark}
\newtheorem{question}[thm]{Question}
\theoremstyle{remark*}
\newtheorem*{question*}{Question} 
\theoremstyle{remark}
\newtheorem{rem}[thm]{\protect\remarkname}
\theoremstyle{remark*}
\newtheorem*{rem*}{\protect\remarkname}
\theoremstyle{remark}

\theoremstyle{remark*}
\newtheorem*{example*}{\protect\examplename}

\theoremstyle{plain}

\providecommand{\corollaryname}{Corollary}
  
\theoremstyle{definition}
  
\makeatother

\theoremstyle{plain} % just in case the style had changed
\newcommand{\thistheoremname}{}
\newtheorem{genericthm}[thm]{\thistheoremname}

\newtheorem*{genericthm*}{\thistheoremname}
\newenvironment{namedthm*}[1]
  {\renewcommand{\thistheoremname}{#1}%
   \begin{genericthm*}}
  {\end{genericthm*}}

\makeatother

\usepackage{babel}
 \providecommand{\lemmaname}{Lemma}
  \providecommand{\propositionname}{Proposition}
  \providecommand{\remarkname}{Remark}
\providecommand{\theoremname}{Theorem}

\newcommand{\R}{\mathbb{R}}
\newcommand{\N}{\mathbb{N}}

\newcommand{\Z}{\mathbb{Z}}

\newcommand\precdot{\mathrel{\ooalign{$\prec$\cr
  \hidewidth\raise0ex\hbox{$\cdot\mkern0.5mu$}\cr}}}
\newcommand\preceqdot{\mathrel{\ooalign{$\preceq$\cr
  \hidewidth\raise0.225ex\hbox{$\cdot\mkern0.5mu$}\cr}}}

\usepackage{fancyhdr}

\title{Failure of Khintchine-type results along the polynomial image of  IP$_0$ sets}
\usepackage{chngcntr}
\usepackage{apptools}
\AtAppendix{\counterwithin{thm}{section}}

\begin{document}
\begin{frontmatter}
\author[add1]{Rigoberto Zelada}
\ead{rzelada@umd.edu}

\address[add1]{Department of Mathematics. University of Maryland, College Park, MD 20742, USA}
    \begin{abstract}
In \cite{BFM}, Bergelson, Furstenberg, and McCutcheon established the following far reaching extension of Khintchine's recurrence theorem:\\
For any invertible probability preserving system $(X,\mathcal A,\mu,T)$, any   non-constant polynomial $p\in\Z[x]$ with $p(0)=0$, any $A\in\mathcal A$, and any $\epsilon>0$, the set 
$$R_\epsilon^p(A)=\{n\in\N\,|\,\mu(A\cap T^{-p(n)}A)>\mu^2(A)-\epsilon\}$$
is IP$^*$, meaning that for any increasing sequence $(n_k)_{k\in\N}$ in $\N$,
$$\text{FS}((n_k)_{k\in\N})\cap R_\epsilon^p(A)\neq \emptyset,$$
where 
$$\text{FS}((n_k)_{k\in\N})=\{\sum_{j\in F}n_j\,|\,F\subseteq \N\,\text{ is finite}\text{ and }F\neq\emptyset\}=\{n_{k_1}+\cdots+n_{k_t}\,|\,k_1<\cdots<k_t,\,t\in\N\}.$$
In view of the potential new applications to combinatorics, this result has led to the question of whether a further strengthening of Khintchine's recurrence theorem holds, namely  whether   the set $R_\epsilon^p(A)$ is IP$_0^*$ meaning that there exists a $t\in\N$ such that for any finite sequence $n_1<\cdots<n_t$ in $\N$, 
$$\{\sum_{j\in F}n_j\,|\,F\subseteq \{1,...,t\}\text{ and }F\neq \emptyset\}\cap R_\epsilon^p(A)\neq \emptyset.$$
In this paper we give a negative answer to this question by showing that  for any given polynomial $p\in\Z[x]$ with $\deg(p)>1$ and $p(0)=0$ there is an invertible probability  preserving system $(X,\mathcal A,\mu,T)$, a set  $A\in\mathcal A$,  and an $\epsilon>0$ for which the set $R_\epsilon^p(A)$ is not IP$_0^*$.
    \end{abstract}
\end{frontmatter}
\textbf{Key words}:  Gaussian systems, sets of recurrence, Ramsey theory.\\
\textbf{MSC classification}: 37A05 (Primary) 37A46; 37A50; 05D10 (Secondary).
%%%%%%%%%%%%%%%%%%%%%%%%%%%%%%%%%%%%%%%%%%%%%%%%%%%%%%%%%
%\tableofcontents
\section{Introduction}
A set $E\subseteq\N=\{1,2,...\}$ is an IP set if there exists an increasing sequence $(n_k)_{k\in\N}$ in $\N$ such that 
$$\text{FS}((n_k)_{k\in\N})\subseteq E,$$
where the set of finite sums $\text{FS}((n_k)_{k\in\N})$ is defined by 
$$\text{FS}((n_k)_{k\in\N}):=\{\sum_{j\in F}n_j\,|\,F\subseteq \N\,\text{ is finite}\text{ and }F\neq\emptyset\}=\{n_{k_1}+\cdots+n_{k_t}\,|\,k_1<\cdots<k_t,\,t\in\N\}.$$
IP sets play a prominent role in the  theory of measurable and topological multiple recurrence. In particular, multiple recurrence theorems along IP sets lead to strong applications in combinatorics some of which, so far, are not achievable by conventional methods (see \cite{FBook}, \cite{FKIPSzemerediLong}, \cite{berMcCuIPPolySzemeredi}).\\
In \cite{BFM}, Bergelson, Furstenberg, and McCutcheon established a far reaching extension of Khintchine's recurrence theorem. (Khintchine's theorem states that for any probability preserving system $(X,\mathcal A,\mu,T)$, any $A\in\mathcal A$, and any $\epsilon>0$, the set of \textit{large returns}
$$\{n\in\N\,|\,\mu(A\cap T^{-n}A)>\mu^2(A)-\epsilon\}$$
is syndetic, meaning that it has bounded gaps.) We state here a special case of a more general result proved in \cite{BFM}. A set $E\subseteq \N$ is called IP$^*$ if it has a non-trivial intersection with every IP set in $\N$. 
\begin{thm}[Cf. Corollary 2.1 in \cite{BFM}]\label{0.BFMThm}
Let $(X,\mathcal A,\mu,T)$ be an invertible probability preserving system and let $p\in\Z[x]$ be a   non-constant polynomial with $p(0)=0$. For any $A\in\mathcal A$ and any $\epsilon>0$, the set 
$$R_\epsilon^p(A):=\{n\in\N\,|\,\mu(A\cap T^{-p(n)}A)>\mu^2(A)-\epsilon\}$$
is \rm{IP$^*$}.
\end{thm}
\begin{rem}
    It is not hard to show that IP$^*$ sets are syndetic (see, for example, \cite[Lemma 9.2]{FBook}). Moreover, IP$^*$ sets have the  finite intersection property: if  $E,F\subseteq\N$ are IP$^*$, then $E\cap F$ is also IP$^*$ \cite[Lemma 9.5]{FBook}.  It follows that IP$^*$ is a stronger property  than syndetic (clearly the family of syndetic sets does not have the finite intersection property) and, hence, \cref{0.BFMThm} is an enhancement of Khintchine's  theorem even when $p(x)$ is linear. We remark in passing that the finite intersection property of IP$^*$ sets is a consequence of Hindman's theorem \cite{HIPPartitionRegular} which implies that if  $S$ is an IP set in $\N$ and $E\cup F=S$, then at least one of  $E$ and $F$ is also an IP set.
\end{rem}
%That \cref{0.BFMThm} extends Khintchine's recurrence  theorem follows from the fact that any IP$^*$ set is syndetic \cite[Lemma 9.2]{FBook}. Furthermore, since, unlike syndetic sets, IP$^*$ sets have the property  that whenever  $E,F\subseteq\N$ are IP$^*$, so is $E\cap F$. As shown in \cite[Lemma 9.5]{FBook}, this intersective property of IP$^*$ sets is a consequence of the Ramsey property of IP sets. As a  matter of fact,  \cite[Lemma 9.5]{FBook} states the following general result which will be relevant in the coming discussion: Let $\mathcal S$ be a class of non-empty subsets of $\N$ and denote by $\mathcal S^*$ the class of subsets of $\N$ having a non-trivial intersection with every member of $\mathcal S$. If $\mathcal S$ has the Ramsey property, then for any $E,F\in\mathcal S^*$, $E\cap F\in \mathcal S^*$.\\

We now define IP$_r$ sets, a finitary version of IP sets. Given $r\in\N$, a set $E\subseteq\N$ is called IP$_r$ if there exist positive integers $n_1<\cdots<n_r$ with 
$$\text{FS}((n_k)_{k=1}^r):=\{\sum_{j\in F}n_j\,|\,F\subseteq \{1,...,r\}\text{ and }F\neq \emptyset\}\,\subseteq E.$$
A set $\Gamma\subseteq \N$ is called an IP$_0$ set if for every $r\in\N$, $\Gamma$ is an IP$_r$ set.\\
 IP$_0$ sets have the \textit{Ramsey property}. Namely, if $S$ is an IP$_0$ set in $\N$ and $E\cup F=S$, then at least one of $E$ and $F$ is IP$_0$. This property of IP$_0$ sets is a consequence of  
the following fact about IP$_r$ sets [Cf. \cite{BDonaldRobertsonIP_r}, Proposition 2.3]: Given $t,r\in\N$ there exists an $R\in\N$ such that if $E\subseteq \N$ is an IP$_R$ set and $\{C_1,...,C_t\}$  forms a partition of $E$, then one of $C_1,...,C_t$ is an IP$_r$ set. It is worth noting that IP$_r$ and IP$_0$ sets naturally appear in the framework of multiple recurrence and combinatorial applications thereof  \cite{FKIPSzemerediLong}, \cite{BerUltraAcrossMath}, \cite{bergelsonShiftedPrimes},  \cite{AlmostIPBerLeib}, \cite{IPrRecNilsystems}.\\

In view of the potential new applications to combinatorics, it is of interest to obtain a "finitary" version of \cref{0.BFMThm} involving IP$_r$ sets. Given $r\in\N$, a set $E\subseteq \N$ is called IP$_r^*$ if it has a non-empty intersection with every IP$_r$ set in $\N$. Similarly, a set $E\subseteq \N$ is called IP$_0^*$ if it has a non-empty intersection with every IP$_0$ set in $\N$. It is easy to see that a set is IP$_0^*$ if and only if it is  IP$_r^*$ for some $r\in\N$. 
%(To see this note that $\Gamma\subseteq\N$ is not an IP$_0^*$ set if and only if there is an IP$_0$ set $E$ with $\Gamma\cap E=\emptyset$ if and only if there is a sequence  $(E_r)_{r\in\N}$ of subsets of $\N$ such that $\Gamma\cap \bigcup_{r\in\N}E_r=\emptyset$ and for each $r\in\N$, $E_r$ is an IP$_r$ set if and only if $\Gamma$ is not an IP$_r^*$ set for each $r\in\N$.)
\begin{question}\label{0.MainQuestion}
    Can  \cref{0.BFMThm} be refined by replacing IP$^*$ with IP$_r^*$? In other words, is it true that for  any non-constant $p\in\Z[x]$ with $p(0)=0$, any invertible  probability preserving system $(X,\mathcal A,\mu, T)$,  any $A\in\mathcal A$, and any $\epsilon>0$, the set $R_\epsilon^p(A)$ is IP$_0^*$ (i.e. $R_\epsilon^p(A)$ is IP$_r^*$ for some $r=r(\epsilon)$)?
\end{question}
In this paper we show  that, in general, \Cref{0.MainQuestion} has a negative answer.  
\begin{namedthm*}{Theorem A}
For any $p\in\Z[x]$ with $p(0)=0$ and $\deg(p)>1$, there exists an invertible probability preserving system $(X,\mathcal A,\mu,T)$, a set $A\in\mathcal A$ with $\mu(A)=\frac{1} {2}$, a  $\delta>0$, and an IP$_0$ set $\Gamma$ such that for any $n\in\Gamma$, 
$$\mu(A\cap T^{-p(n)}A)<\mu^2(A)-\delta.$$
\end{namedthm*}
Let $(X,\mathcal A,\mu, T)$ be an invertible probability preserving system. By the traditional
abuse of notation, we denote the unitary operator induced by $T$ on $L^2(\mu)$ by $T$.
Thus, $T f = f  \circ T$ for any $f$ in $L^2(\mu)$.
As we will show in Section 5 below, the invertible probability preserving system $(X,\mathcal A,\mu, T)$ in the formulation of  Theorem A can be chosen  to be both weakly mixing (i.e. $T:L^2(\mu)\rightarrow L^2(\mu)$ has no non-trivial eigenvectors) and rigid (i.e. for each $f\in L^2(\mu)$, there exists an increasing sequence $(n_k)_{k\in\N}$ with $\lim_{k\rightarrow\infty}T^{n_k}f=f$ in the $L^2$-norm):
\begin{namedthm*}{Proposition B}
The invertible probability preserving system $(X,\mathcal A,\mu, T)$ in the statement of Theorem A can be picked to be both   weakly mixing  and rigid. 
\end{namedthm*}
We would like to mention that the following question which is a special case of  a question formulated  in 
 \cite[p. 41]{BerUltraAcrossMath} and generalizes \Cref{0.MainQuestion}, is still open. (In its full generality, this question in  \cite[p. 41]{BerUltraAcrossMath} deals with \textit{multiple} recurrence and was originally motivated by \cite{FKIPSzemerediLong} and \cite{berMcCuIPPolySzemeredi}.) 
 \begin{question}\label{0.WeakIP_rKhintchine}
     Is it true that for any non-constant $p\in\Z[x]$ with $p(0)=0$, any invertible probability preserving system $(X,\mathcal A,\mu,T)$, and any $A\in\mathcal A$ with $\mu(A)>0$, there exists a $c>0$ for which the set 
     $$R=\{n\in\N\,|\,\mu(A\cap T^{-p(n)}A)>c\}$$
     is IP$_0^*$?
 \end{question}
To make the picture more complete we now state three  results which provide conditions 
under which Questions \ref{0.MainQuestion} and \ref{0.WeakIP_rKhintchine} have an affirmative answer. We remark that, to some extent, these results made it tempting to believe that the answer to \Cref{0.MainQuestion} is positive regardless of the properties of the transformation  $T$ and the degree of  $p\in\Z[x]$.\\

The first of these results, \cref{0.DiscreteSpectrum} below, deals with invertible $\mu$-measure preserving transformations  $T$ having discrete spectrum  (meaning that $L^2(\mu)$ is spanned by the eigenvectors of $T$) and it follows from  Theorem 7.7 in \cite{BerUltraAcrossMath}. 
\begin{thm}\label{0.DiscreteSpectrum}
    Let $(X,\mathcal A,\mu)$ be a probability space and let $p\in\Z[x]$ be a non-constant polynomial with $p(0)=0$.  Suppose that $T:X\rightarrow X$ is an invertible $\mu$-preserving transformation with  discrete spectrum. For any $A\in\mathcal A$ and any $\epsilon>0$, the set
    \begin{equation}\label{0.DiscreteSpectrumEq}
    \{n\in\N\,|\,\mu(A\cap T^{-p(n)}A)>\mu(A)-\epsilon\}
    \end{equation}
     is \rm{IP$_0^*$}.
\end{thm}
\begin{rem}
    By Theorem A, the assumption that $T$ has discrete spectrum in \cref{0.DiscreteSpectrum} cannot be removed (even if one replaces $\mu(A)$ with $\mu^2(A)$ in formula \eqref{0.DiscreteSpectrumEq} above). Also note that \cref{0.DiscreteSpectrum} implies that $(X,\mathcal A,\mu,T)$ in Theorem A cannot have discrete spectrum.
\end{rem}

The following theorem can be viewed as one more variant  of Khintchine's recurrence theorem; it follows from \cite[Theorem 2.1]{AlmostIPBerLeib}.
\begin{thm}\label{0.IP0Khintchine}
    Let $(X,\mathcal A,\mu,T)$ be an invertible probability preserving system. For any $A\in\mathcal A$, any $d\in\N$, and any $\epsilon>0$, the set 
    \begin{equation}\label{0.IP0KhinthchineEq}
    R_{\epsilon,d}(A):=\{n\in\N\,|\,\mu(A\cap T^{-dn}A)>\mu^2(A)-\epsilon\}
    \end{equation}
    is \rm{IP$_0^*$}.
\end{thm}
\begin{rem}
As a matter of fact, Theorem 2.1 in \cite{AlmostIPBerLeib} states that the sets of the form $R_{\epsilon,d}(A)$ have a property even stronger than that of IP$_0^*$: they are $\Delta_0^*$. Given $r>1$, a set $E\subseteq \N$ is a $\Delta_r$ set if there exist positive integers $n_1<\cdots<n_r$ with 
$$\{n_j-n_i\,|\,1\leq i<j\leq r\}\subseteq E.$$
A set $E\subseteq \N$ is called a $\Delta_r^*$ set if it has a non-empty intersection with every $\Delta_r$ set in $\N$. A set $E\subseteq \N$ is called $\Delta_0^*$ if it is a $\Delta_r^*$ set  for some $r\in\N$. By a  modification of an argument presented in  \cite[p. 177]{FBook}, one can show that for every $p\in\Z[x]$ with $p(0)=0$ and  $\deg(p)>1$ the sets of the form 
$$\{n\in\N\,|\,\mu(A\cap T^{-p(n)}A)>0\}$$
 are not necessarily $\Delta_0^*$ (see also \cite[Theorem 1.4, case $\ell=1$]{BerZel_JCTA_iteratedDifferences2021}). Thus, the variant of \Cref{0.WeakIP_rKhintchine} dealing with $\Delta_0^*$ sets instead of IP$_0^*$ has a negative answer.
\end{rem}
The next results demonstrates that, roughly speaking, the sets of the form $R_\epsilon^p(A)$ are "almost" IP$_0^*$. For a finite set  $F$, we denote its cardinality by $|F|$.
\begin{thm}[Cf. Theorem 1.8 in \cite{AlmostIPBerLeib}]
Let $(X,\mathcal A,\mu,T)$ be an invertible ergodic probability preserving system and let $p\in\Z[x]$ be a   non-constant polynomial with $p(0)=0$. For any $A\in\mathcal A$ with $\mu(A)>0$ and any $\epsilon>0$, there exists a set $E\subseteq \N$ with 
$$\lim_{N-M\rightarrow\infty}\frac{|E\cap \{M+1,...,N\}|}{N-M}=0$$
and such that $E\cup R_\epsilon^p(A)$
is \rm{IP$^*_0$}.
\end{thm}
We conclude this introduction by stating various results each of which can be viewed as the unitary counterpart of one of the results mentioned above. Let  $\mathcal H$ be a (separable) Hilbert space and for any set of bounded linear operators $B$, let $\overline{B^{\text{WOT}}}$ denote the closure of $B$ with respect to the weak operator topology. 
\begin{enumerate}
    \item [-] (Cf. \cref{0.BFMThm} above) Let $U:\mathcal H\rightarrow\mathcal H$ be a unitary operator, let $p\in\Z[x]$ be a non-constant polynomial with $p(0)=0$, and let $E\subseteq \N$ be an IP set. There exists $V\in\overline{\{U^{p(n)}\,|\,n\in E\}^{WOT}}$ such that for any $\xi\in\mathcal H$, $\langle V\xi,\xi\rangle\geq 0$. As a matter of fact, an immediate consequence of  Theorem 1.8 in \cite{BFM} is that one can take $V$ to satisfy $V\circ V=V$ and $V=V^*$.
    \item [-] (Cf. \cref{0.IP0Khintchine} above) Let $U:\mathcal H\rightarrow\mathcal H$ be a unitary operator and let $E\subseteq \N$ be an IP$_0$ set. There exists $V\in\overline{\{U^{n}\,|\,n\in E\}^{WOT}}$ such that for any $\xi\in\mathcal H$, $\text{Re}(\langle V\xi,\xi\rangle)\geq 0$.
    \item [-] (Cf. Theorem A above) Let $p\in\Z[x]$ be a non-constant polynomial with $p(0)=0$ and $\deg(p)>1$. There exists an IP$_0$ set $\Gamma\subseteq \N$, a unitary operator $U:\mathcal H\rightarrow\mathcal H$, a $\delta>0$, and a $\xi\in\mathcal H$ such that for any $n\in\Gamma$, $\langle U^{p(n)}\xi,\xi\rangle<-\delta$.
\end{enumerate}
%%%%%%%%%%%%%%%%%%%%%%%%%%%55
The structure of this paper is as follows. In Section 2 we review the necessary background on Gaussian systems. In Section 3 we prove Theorem A with the help of \cref{4.SpectralMeasure}, our main auxiliary result. In Section 4, we prove \cref{4.SpectralMeasure}. In Section 5 we prove Proposition B. \\

\textbf{Acknowledgments:} The author would like to thank Professor Vitaly Bergelson for his valuable input during the preparation of this manuscript. 
%%%%%%%%%%%%%%%%%%%%%%%%%%%%%%%%%%%%%%%%%%%%%%%%%%%%%%%%%%%%%%%%%%%%%%%%%%%%%%%%%%
\section{Background on Gaussian systems}
In this section we review the necessary background on  a special type of probability preserving systems  known as  \textbf{Gaussian systems}.  For the construction of a Gaussian system see 
 \cite[Chapter 8]{cornfeld1982ergodic} or \cite[Appendix C]{kechris2010Global}, for example. \\

Let $\mathcal A=\text{Borel}(\R^\Z)$ and  let $T:\R^{\Z}\rightarrow\R^{\Z}$ be defined by 
$$T(\omega)(n)=\omega(n+1).$$
For any given Borel probability measure $\rho$ on $\mathbb T$ let $\phi=\phi_\rho:\Z\rightarrow \mathbb C$ be defined by 
$$\phi(n)=\int_\mathbb T e^{2\pi inx}\text{d}\rho(x).$$
Suppose that $\phi(n)=\phi(-n)$ for each $n\in\Z$. Then 
 there exists a $T$-invariant  probability measure $\gamma$ on $\mathcal A$ with the following properties:
\begin{enumerate}
\item [G.1] For any given $n\in\Z$, let $X_{n}:\R^\Z\rightarrow \R$ be the canonical projection defined by $X_{n}(\omega)=\omega(n)$. The random variable $X_{n}$ has normal distribution with mean zero and variance $1$. In other words, for any Borel measurable $A\subseteq \R$ and any $n\in\Z$,
\begin{equation}\label{GaussianMarginal}
\gamma(X_n^{-1}A)=\frac{1}{\sqrt{2\pi}}\int_Ae^{-\frac{x^2}{2}}\text{d}x.
\end{equation}
\item [G.2] For any  $m,n\in\Z$, 
$$\int_{\R^\Z}X_{m}X_{n}\text{d}\gamma=\phi(m-n).$$
\item [G.3] Given any $m,n\in\Z$, let 
$$C=\begin{pmatrix}
1 & \phi(m-n)\\
\phi(n-m) & 1
\end{pmatrix}$$
(We remark that $\phi(m-n)=\phi(n-m)\in\R$ and, hence, $C$ is symmetric and real-valued). If $\det (C)\neq 0$, then 
$$\gamma(X^{-1}_{m}(0,\infty)\cap X^{-1}_{n}(0,\infty))=\frac{1}{2\pi\sqrt{\det(C)}}\int_{0}^\infty\int_0^\infty e^{-\frac{1}{2} (x,y)C^{-1}(x,y)^T}\text{d}x\text{d}y.$$
\item [G.4] (Cf. Theorem 2.3 item (i) in \cite{zelada2022GaussianETDS}) For any sequence $(n_k)_{k\in\N}$ in $\Z$ one has that  for every $A,B\in\mathcal A$,
$$\lim_{k\rightarrow\infty}\gamma(A\cap T^{-n_k}B)=\gamma(A)\gamma(B)$$
 if and only if  for every $m\in\Z$,
 $$\lim_{k\rightarrow\infty}\phi(m+n_k)=0.$$
 \item [G.5] (Cf. Theorem 2.3 item (ii) in \cite{zelada2022GaussianETDS}) For any sequence $(n_k)_{k\in\N}$ in $\Z$ one has that  for every $A,B\in\mathcal A$,
$$\lim_{k\rightarrow\infty}\gamma(A\cap T^{-n_k}B)=\gamma(A\cap B)$$
 if and only if  for every $m\in\Z$,
 $$\lim_{k\rightarrow\infty}\phi(m+n_k)=\phi(m).$$
\end{enumerate}
The invertible probability preserving system $(\R^{\Z},\mathcal A,\gamma, T)$ is called the \textbf{Gaussian system} associated with $\rho$.
\begin{rem}
    One can also define the Gaussian system associated with any finite, positive Borel measure defined on $\mathbb T$. The properties of such a Gaussian system are analogous to those of a Gaussian system associated with a probability measure. For more on Gaussian systems see \cite{cornfeld1982ergodic} or \cite{kechris2010Global}, for example. 
\end{rem}
The following lemma will be utilized in the sequel. For any $\delta \in (-1,1)$ we will let 
$$\Sigma_\delta=\begin{pmatrix}
    1 & \delta\\
    \delta & 1
\end{pmatrix}$$
Observe that 
\begin{equation}\label{5.eq:SigmaInverse}
\Sigma^{-1}_\delta=\frac{1}{1-\delta^2}\begin{pmatrix}
    1 & -\delta \\
    -\delta & 1
\end{pmatrix}.
\end{equation}
\begin{lem}\label{3.Lem:GaussianIntegral}
    Let $\delta\in (-1,1)$. Then 
    $$\frac{1}{2\pi \sqrt{1-\delta^2}}\int_0^\infty\int_0^\infty e^{-\frac{1}{2}(x,y)\Sigma_\delta^{-1}(x,y)^T}\text{d}x\text{d}y=\frac{1}{4}+\frac{\sin^{-1}(\delta)}{2\pi}.$$
    So, in particular, if $\delta<0$, 
     $$\frac{1}{2\pi \sqrt{1-\delta^2}}\int_0^\infty\int_0^\infty e^{-\frac{1}{2}(x,y)\Sigma_\delta^{-1}(x,y)^T}\text{d}x\text{d}y<\frac{1}{4}.$$
\end{lem}
\begin{proof}
    By \eqref{5.eq:SigmaInverse},
    \begin{multline*}
   \frac{1}{2\pi \sqrt{1-\delta^2}} \int_0^\infty\int_0^\infty e^{-\frac{1}{2}(x,y)\Sigma_\delta^{-1}(x,y)^T}\text{d}x\text{d}y=\frac{1}{2\pi \sqrt{1-\delta^2}}\int_0^\infty\int_0^\infty e^{-\frac{x^2-2\delta xy+y^2}{2(1-\delta^2)}}\text{d}x\text{d}y\\
   =\frac{1}{2\pi \sqrt{1-\delta^2}}\int_0^\infty\int_0^\infty e^{-\frac{x^2}{2}}e^{-\frac{(y-\delta x)^2}{2(1-\delta^2)}}\text{d}x\text{d}y
    \end{multline*}
    Thus, substituting $u=\frac{y-\delta x}{\sqrt{1-\delta^2}}$ and then switching to polar coordinates yields, 
    \begin{equation*}
    \frac{1}{2\pi}\int_0^\infty \int_{-\delta x/\sqrt{1-\delta^2}}^\infty e^{-\frac{x^2+u^2}{2}}\text{d}u\text{d}x
         =\frac{1}{2\pi}\int_{\sin^{-1}(-\delta)}^{\pi/2}\int_{0}^\infty e^{-\frac{r^2}{2}}r\text{d}r\text{d}\theta=\frac{1}{4}+\frac{\sin^{-1}(\delta)}{2\pi}.
    \end{equation*}
    We are done.
\end{proof}
%%%%%%%%%%%%%%%%%%%%%%%%%%%%%%%%%%%%%%%%%%%%%%%%%%%%%
\section{Proof of Theorem A}
We will prove the following stronger version of Theorem A which provides examples involving not only IP$_0$ sets but also other additive structures with interesting recurrence properties (e.g. the so called $\Delta_{\ell,r}$ sets introduced in \cite{BerZel_JCTA_iteratedDifferences2021}).
\begin{namedthm*}{Theorem A$^\prime$}\label{3.TheoremA'}
For any $p\in\Z[x]$ with $p(0)=0$ and $\deg(p)>1$, there exist a Gaussian system  $(\R^\Z,\mathcal A,\gamma,T)$, a set $A\in\mathcal A$ with $\gamma(A)=\frac{1} {2}$, and a family of sequences $((n_{t,s})_{s=1}^t)_{t\in\N}$ in $\N$ such that for any $\epsilon>0$, there exists a $t_\epsilon\in\N$ with the property that  for any $t\geq t_\epsilon$, any $s_1,...,s_\ell\in\{1,...,t\}$ with $s_1<\cdots<s_\ell$, and any $\xi_1,...,\xi_\ell\in\{-1,1\}$,
 $$\gamma(A\cap T^{-p(\sum_{j=1}^\ell\xi_jn_{t,s_j})}A)<\frac{1}{6}+\epsilon.$$
\end{namedthm*}
To prove Theorem A$^\prime$ we will need the following lemma the proof of which will be provided in the next section.. For a finite set  $F$, we denote its cardinality  by $|F|$.
\begin{lem}\label{4.SpectralMeasure}
    Let $p\in\Z[x]$ be such that  $p(0)=0$ and $\deg(p)>1$. There exists a family of finite sequences $((m_{t,s})_{s=1}^{t^2})_{t\in\N}$ in $\N$ and a Borel probability measure $\mu$ on $\mathbb T$ with the following property
\begin{equation}\label{3.eq:NegativeCorrelations}
\lim_{t\rightarrow\infty}\left(\max_{F\subseteq \{1,...,t^2\},\,|F|\geq t}\max_{\xi_1,...,\xi_{t^2}\in\{-1,1\}}\left|\int_\mathbb T e^{2\pi ip(\sum_{j\in F}\xi_jm_{t,j})x}\text{d}\mu(x)+\frac{1}{2}\right|\right)=0.
\end{equation}
\end{lem}
\begin{proof}[Proof of Theorem A$^\prime$]
    Let $\mu$ and $((m_{t,s})_{s=1}^{t^2})_{t\in\N}$ be as in in the statement of \cref{4.SpectralMeasure} and denote by $\rho$ the unique probability measure with the property that 
    $$\int_\mathbb T e^{2\pi inx}\text{d}\rho(x)=\frac{1}{2}\left(\int_\mathbb T e^{2\pi inx}\text{d}\mu(x)+\int_\mathbb T e^{-2\pi inx}\text{d}\mu(x)\right)$$
    for each $n\in\Z$. Observe that $\rho$ is a Borel probability measure on $\mathbb T$ and for each $n\in\Z$,
    $$\int_\mathbb T e^{2\pi inx}\text{d}\rho(x)=\int_\mathbb T e^{-2\pi inx}\text{d}\rho(x).$$
    Let  $(\R^\Z,\mathcal A,\gamma, T)$ be the Gaussian system associated with $\rho$. For each $t\in\N$ and each $s\in\{1,...,t\}$, let 
    $$n_{t,s}=\sum_{j=1}^t m_{t,(s-1)t+j}.$$
    Note that, by \eqref{3.eq:NegativeCorrelations}, for any $r>0$ there exists a $t_r\in\N$ such that for any $t\geq t_r$,
    \begin{multline}\label{5.CloseTo-0.5}
    \max_{F\subseteq \{1,...,t\},\,F\neq\emptyset}\max_{\xi_1,...,\xi_{t}\in\{-1,1\}}\left|\int_\mathbb T e^{2\pi ip(\sum_{j\in F}\xi_jn_{t,j})x}\text{d}\rho(x)+\frac{1}{2}\right|\\
    %%%%%%%%%%%%%%%%%%%%%%%%%%%%%%%%%%%%%%%%%
    \leq\max_{F\subseteq \{1,...,t\},\,F\neq\emptyset}\max_{\xi_1,...,\xi_{t}\in\{-1,1\}}\hfill\\
    \frac{1}{2}\left(\left|\int_\mathbb T e^{2\pi ip(\sum_{j\in F}\xi_jn_{t,j})x}\text{d}\mu(x)+\frac{1}{2}\right|+\left|\int_\mathbb T e^{-2\pi ip(\sum_{j\in F}\xi_jn_{t,j})x}\text{d}\mu(x)+\frac{1}{2}\right|\right)\\
    %%%%%%%%%%%%%%%%%%%%%%%%%%%%%%%%%
    =\max_{F\subseteq \{1,...,t\},\,F\neq\emptyset}\max_{\xi_1,...,\xi_{t}\in\{-1,1\}}\hfill\\
    \frac{1}{2}\left(\left|\int_\mathbb T e^{2\pi ip(\sum_{j\in F}\xi_jn_{t,j})x}\text{d}\mu(x)+\frac{1}{2}\right|+\left|\overline{\int_\mathbb T e^{2\pi ip(\sum_{j\in F}\xi_jn_{t,j})x}\text{d}\mu(x)}+\frac{1}{2}\right|\right)\\
    %%%%%%%%%%%%%%%%%%%%%%%%%%%%%%%%%%%%%%%%%%%%%%%%
    =\max_{F\subseteq \{1,...,t\},\,F\neq\emptyset}\max_{\xi_1,...,\xi_{t}\in\{-1,1\}}\left|\int_\mathbb T e^{2\pi ip(\sum_{j\in F}\xi_jn_{t,j})x}\text{d}\mu(x)+\frac{1}{2}\right|
    \leq r.
    \end{multline}
    Also, by \cref{3.Lem:GaussianIntegral}, 
    \begin{equation}\label{5.LimitOfmeasures}
    \lim_{\delta\rightarrow-\frac{1}{2}}\frac{1}{2\pi \sqrt{1-\delta^2}}\int_0^\infty\int_0^\infty e^{-\frac{1}{2}(x,y)\Sigma_\delta^{-1}(x,y)^T}\text{d}x\text{d}y=\lim_{\delta\rightarrow-\frac{1}{2}}\left(\frac{1}{4}+\frac{\sin^{-1}(\delta)}{2\pi}\right)=\frac{1}{4}-\frac{1}{12}=\frac{1}{6}.
    \end{equation}
    Set $A=X_0^{-1}(0,\infty)$. Combining (G.3), \eqref{5.CloseTo-0.5}, and \eqref{5.LimitOfmeasures}, we see that for any $\epsilon>0$, there exists a $t_\epsilon\in\N$ such that for any $t\geq t_\epsilon$, any $s_1,...,s_\ell\in\{1,...,t\}$ with $s_1<\cdots<s_\ell$, and any $\xi_1,...,\xi_\ell\in\{-1,1\}$,
    \begin{multline*}
    |\gamma(A\cap T^{-p(\sum_{j=1}^\ell\xi_jn_{t,s_j})}A)-\frac{1}{6}|=|\gamma\left(X_0^{-1}(0,\infty)\cap X_{p(\sum_{j=1}^\ell\xi_jn_{t,s_j})}^{-1}(0,\infty)\right)-\frac{1}{6}|\\
    =|\frac{1}{2\pi \sqrt{1-\delta^2}}\int_0^\infty\int_0^\infty e^{-\frac{1}{2}(x,y)\Sigma_\delta^{-1}(x,y)^T}\text{d}x\text{d}y-\frac{1}{6}|<\epsilon,
    \end{multline*}
    where
    $$\delta=\int_\mathbb T e^{2\pi ip(\sum_{j=1}^\ell\xi_jn_{t,s_j})x}\text{d}\rho(x).$$
    Thus, 
    $$\gamma(A\cap T^{-p(\sum_{j=1}^\ell\xi_jn_{t,s_j})}A)<\frac{1}{6}+\epsilon.$$
    Theorem A$^\prime$ now follows by noting that, by \eqref{GaussianMarginal}, $\gamma(A)=\frac{1}{2}$. 
\end{proof}
%%%%%%%%%%%%%%%%%%%%%%%%%%%%%%%%%%%%%%%%%%%%%%%%%%5
\section{Proof of \cref{4.SpectralMeasure}}
To prove \cref{4.SpectralMeasure}, we utilize a strategy similar to the one used in \cite[Section 3]{zelada2022GaussianETDS}. This strategy consists in constructing  a measurable map $f$ from a Cantor set $\mathcal C$ to $\mathbb T$ and a probability measure $\mathbb P$ on $\mathcal C$ such that the measure 
$$\mu=\mathbb P\circ f^{-1}$$
is the measure $\mu$ described in \cref{4.SpectralMeasure}. In Subsection 4.1 we will establish various  estimates that will be needed  for the construction of $f$. In Subsection 4.2 we will define convenient random variables that will be later used for the construction of $\mathbb P$. In Subsection 4.3 we will prove  \cref{4.SpectralMeasure}.
%%%%%%%%%%%%%%%%%%%%%%%%%%%%%%%%%%%%%%%%%%%%%%%%%%%%%%%%%%%%%%%%%
\subsection{Some relevant estimates}
Through this subsection we will fix $b,d\in\N$, $d>1$, and let $((m_{t,s})_{s=1}^{t^2})_{t\in\N}$ be a family of finite sequences in $\N$ with the following properties:
\begin{enumerate}
\item [(C.1)] For any $t\in\N$, $m_{t,1}>t^{2d+4}2^{t+1}2^d$.
\item [(C.2)] For any $t\in\N$ and any $s\in\{1,...,t^2\}\setminus\{t^2\}$, 
$$(2bdm_{t,s}^{d+1})|m_{t,(s+1)}.$$ 
\item [(C.3)] For any $t\in\N$,
$$(2bdm_{t,t^2}^{d+1})|m_{(t+1),1}.$$ 
\item [(C.4)] For any $t\in\N$ and any $s,s'\in\{1,...,t^2\}$ with $s<s'$, 
$$\frac{m^d_{t,s}}{m_{t,s'}}<\frac{1}{t^{2d+4}2^{t+1}2^d}.$$
\item [(C.5)] For any $\tau,t\in \N$ with $\tau<t$, any $s\in\{1,...,\tau^2\}$, and any $s'\in\{1,...,t^2\}$,
$$\frac{m^d_{\tau,s}}{m_{t,s'}}<\frac{1}{t^{2d+4}2^{t+1}2^d}.$$
\end{enumerate}
\begin{rem}
    One can construct a sequence  of sequences $((m_{t,s})_{s=1}^{t^2})_{t\in\N}$ satisfying conditions  (C.1)-(C.5) as follows: First, pick a (strictly) increasing sequence  $(m_k)_{k\in\N}$ in $\N$  with the property that for each $k\in\N$, 
    $$(2bdm_k^{d+1})|m_{k+1}$$
    and observe that any subsequence of $(m_k)_{k\in\N}$  inherits this divisibility property. Next, let
     $(m'_{k})_{k\in\N}$ be a subsequence of $(m_k)_{k\in\N}$ with the additional properties that (a) for each $k\in\N$, 
    $$\frac{(m_k')^d}{m'_{k+1}}<\frac{1}{(k+1)^{2d+4}2^{k+2}2^d}$$
    and (b) for each $t\in\N$,
    $$m'_{\sigma_t+1}>t^{2d+4}2^{t+1}2^d,$$
    where, for each $t\in\N$, $\sigma_t=\sum_{r=0}^{t-1}r^2$. 
    Letting $m_{t,s}=m'_{\sigma_t+s}$
    for each $t\in\N$ and each $s\in\{1,...,t^2\}$, we see that 
     $((m_{t,s})_{s=1}^{t^2})_{t\in\N}$ satisfies (C.1)-(C.5).
\end{rem}
Our first estimate is an easy consequence of the rate of growth of $((m_{t,s})_{s=1}^{t^2})_{t\in\N}$.
\begin{lem}\label{2.lem:BoundOftheFuture}
    For any $\tau\in\N$,    $$\sum_{t=\tau+1}^\infty\left(\sum_{s=1}^{t^2}\frac{1}{2m_{t,s}^d}+\sum_{s'=2}^{t^2}\sum_{s=1}^{s'-1}\frac{1}{2dm_{t,s}m_{t,s'}^{d-1}}\right)<\frac{1}{\tau^{2d}m_{\tau,\tau^2}^d2^\tau}$$
\end{lem}
\begin{proof}
    By condition (C.2) we have that for any $t\in\N$, 
    $$m_{t,1}\leq \cdots\leq m_{t,t^2}.$$
    So, for each $t>1$,
    $$\sum_{s=1}^{t^2}\frac{1}{2m_{t,s}^d}+\sum_{s'=2}^{t^2}\sum_{s=1}^{s'-1}\frac{1}{2dm_{t,s}m_{t,s'}^{d-1}}<\sum_{s=1}^{t^2}\frac{1}{m_{t,1}^d}+\sum_{s'=2}^{t^2}\sum_{s=1}^{s'-1}\frac{2}{m_{t,1}^{d}}=\frac{t^4}{m_{t,1}^d}.$$
    It now follows from condition (C.5) that 
    \begin{multline*}
    \sum_{t=\tau+1}^\infty\left(\sum_{s=1}^{t^2}\frac{1}{2m_{t,s}^d}+\sum_{s'=2}^{t^2}\sum_{s=1}^{s'-1}\frac{1}{2dm_{t,s}m_{t,s'}^{d-1}}\right)<\sum_{t=\tau+1}^\infty \frac{t^4}{m_{t,1}^d}\leq\sum_{t=\tau+1}^\infty \frac{t^4}{m_{t,1}}\\
    <\sum_{t=\tau+1}^\infty\frac{t^4}{t^{2d+4}2^tm_{\tau,\tau^2}^d}<\frac{1}{\tau^{2d}m_{\tau,\tau^2}^d}\sum_{t=\tau+1}^\infty\frac{1}{2^t}=\frac{1}{\tau^{2d}m_{\tau,\tau^2}^d2^\tau}
    \end{multline*}
\end{proof}
The next two results follow from the divisibility properties of $((m_{t,s})_{s=1}^{t^2})_{t\in\N}$.
\begin{lem}\label{2.lem:PastPiecesAreIntegers}
    For any $\tau,t\in\N$ with $\tau>t$, any $s\in\{1,...,\tau^2\}$, and any $s',s''\in\{1,...,t^2\}$ with $s'\leq s''$,
    $$\frac{m_{\tau,s}}{2bdm_{t,s'}m_{t,s''}^{d-1}}\in\Z$$
\end{lem}
\begin{proof}
    This result follows immediately from conditions (C.2) and (C.3).
\end{proof}
\begin{lem}\label{2.lem:PastInSamePiece}
    For any $t\in\N$ and any $s,s',s''\in\{1,...,t^2\}$ with $ s'\leq s''<s,$
    $$\frac{m_{t,s}}{2bdm_{t,s'}m_{t,s''}^{d-1}}\in\Z.$$
\end{lem}
\begin{proof}
    This result follows immediately form condition (C.2).
\end{proof}
The last two results in this subsection (see \cref{2.lem:OffDiagonalValuesPresent} and \cref{2.lem:DiagonalValuesPresent} below) are the key for the proof of \cref{4.SpectralMeasure}. For any $r\in \R$, we denote the distance from $r$ to the closest integer by $\|r\|$ (i.e. $\|r\|=\inf_{n\in\Z}|r-n|$). We remark that $\|r\|\leq |r|$.
\begin{lem}\label{2.lem:OffDiagonalValuesPresent}
    Let $t\in\N$ with $t>1$,  $c\in\{1,...,d\}$, $\ell\in\{1,...,t^2\}$, $\xi_{1},...,\xi_{t^2}\in\{-1,1\}$, and $s,s',s_1,...,s_\ell\in\{1,...,t^2\}$ with $s<s'$ and, when $\ell>1$,  $s_1<\cdots<s_\ell$ . Then, 
    \begin{equation}\label{2.eq:CloseTo1/2OffDiagonal}
    \left\|\frac{(\sum_{j=1}^\ell\xi_{s_j} m_{t,s_j})^c}{2bdm_{t,s}m_{t,s'}^{d-1}}-\frac{\xi_{s}\xi_{s'}^{d-1}}{2b}\right\|<\frac{1}{t^42^{t+1}}
    \end{equation}
    whenever the conditions (i) $\ell>1$, (ii) there exist distinct  $i,j\in\{1,...,\ell\}$ with $s_i=s$ and $s_j=s'$, and (iii) $c=d$ hold and 
    \begin{equation}\label{2.eq:CloseTo0OffDiagonal}
    \left\|\frac{(\sum_{j=1}^\ell\xi_{s_j} m_{t,s_j})^c}{2bdm_{t,s}m_{t,s'}^{d-1}}\right\|<\frac{1}{t^42^{t+1}}
    \end{equation}
    otherwise.
\end{lem}
\begin{proof}
    If $s_\ell>s'$, \cref{2.lem:PastInSamePiece} implies that 
    \begin{equation}\label{2.IterativeSimplificationI}
    \frac{(\sum_{j=1}^\ell\xi_{s_j} m_{t,s_j})^c}{2bdm_{t,s}m_{t,s'}^{d-1}}\equiv\frac{\sum_{m=0}^c\binom{c}{m}(\sum_{j=1}^{\ell-1} \xi_{s_j}m_{t,s_j})^{c-m}(\xi_{s_\ell}m_{t,s_\ell})^m}{2bdm_{t,s}m_{t,s'}^{d-1}}\equiv \frac{(\sum_{j=1}^{\ell-1} \xi_{s_j}m_{t,s_j})^c}{2bdm_{t,s}m_{t,s'}^{d-1}}\mod 1,
    \end{equation}
    when $\ell>1$ and 
    \begin{equation}\label{2.IterativeSimplification2}
    \frac{(\sum_{j=1}^\ell \xi_{s_j}m_{t,s_j})^c}{2bdm_{t,s}m_{t,s'}^{d-1}}\equiv \frac{ (\xi_{s_{\ell}}m_{t,s_\ell})^c}{2bdm_{t,s}m_{t,s'}^{d-1}}\equiv 0\mod 1
    \end{equation}
    when $\ell=1$. By iterating \eqref{2.IterativeSimplificationI} and then applying \eqref{2.IterativeSimplification2}, we see that \eqref{2.eq:CloseTo0OffDiagonal} holds whenever $s_1>s'$. Furthermore, by \eqref{2.IterativeSimplificationI}, whenever $s_1\leq s'$, we can assume without loss of generality that $s_\ell\leq s'$. Thus, in order to prove \cref{2.lem:OffDiagonalValuesPresent}, it only remains to show that \eqref{2.eq:CloseTo1/2OffDiagonal} and \eqref{2.eq:CloseTo0OffDiagonal} hold when $s_\ell\leq s'$. 
    We will do this by checking the following cases: (i) $s_\ell<s'$, (ii) $s_\ell=s'$ and $c\leq d-1$, (iii) $s_\ell=s'$, $c=d$, and $\ell=1$, and (iv) $s_\ell=s'$, $c=d$, and $\ell>1$.\\
    
    Suppose first that $s_\ell<s'$. Note that by condition (C.2), $m_{t,1}<m_{t,2}<\cdots<m_{t,t^2}$. So, by condition (C.4),
    $$\left|\frac{(\sum_{j=1}^\ell \xi_{s_j}m_{t,s_j})^c}{2bdm_{t,s}m_{t,s'}^{d-1}}\right|<\frac{(\ell m_{t,s_\ell})^c}{m_{t,s'}}<\frac{t^{2d}m_{t,s_\ell}^d}{m_{t,s'}}<\frac{1}{t^42^{t+1}}.$$
    Next we assume that $s_\ell= s'$ and $c\leq d-1$. By condition (C.1) we obtain
    $$\left|\frac{(\sum_{j=1}^\ell \xi_{s_j}m_{t,s_j})^c}{2bdm_{t,s}m_{t,s'}^{d-1}}\right|<\frac{(\ell m_{t,s'})^{d-1}}{m_{t,s}m_{t,s'}^{d-1}}<\frac{t^{2d}}{m_{t,s}}<\frac{1}{t^42^{t+1}}.$$
    We now let  $s_\ell=s'$, $c=d$, and $\ell=1$. We have 
    $$\frac{(\sum_{j=1}^\ell \xi_{s_j}m_{t,s_j})^c}{2bdm_{t,s}m_{t,s'}^{d-1}}=\frac{\xi^d_{s'}m_{t,s'}^d}{2bdm_{t,s}m_{t,s'}^{d-1}}\in\Z.$$
    Finally, we take $s_\ell=s'$, $c=d$, and $\ell>1$. Observe that 
    $$(\sum_{j=1}^\ell \xi_{s_j}m_{t,s_j})^d=\sum_{m=0}^d\binom{d}{m}(\sum_{j=1}^{\ell-1} \xi_{s_j}m_{t,s_j})^{d-m}(\xi_{s'}m_{t,s'})^m.$$
    When $m<d-1$,
    \begin{equation}\label{2.eq:c=d,m<d-1}
    \left|\frac{\binom{d}{m}(\sum_{j=1}^{\ell-1}\xi_{s_j}m_{t,s_j})^{d-m}(\xi_{s'}m_{t,s'})^{m}}{2bdm_{t,s}m_{t,s'}^{d-1}}\right|<\binom{d}{m}\frac{t^{2d}m_{t,s_{\ell-1}}^d}{m_{t,s'}}<\binom{d}{m}\frac{1}{t^42^{t+1}2^d}.
    \end{equation}
    When $m=d$, 
    \begin{equation}\label{2.eq:c=d,m=d}
    \left|\frac{\binom{d}{m}(\sum_{j=1}^{\ell-1}\xi_{s_j}m_{t,s_j})^{d-m}(\xi_{s'}m_{t,s'})^{m}}{2bdm_{t,s}m_{t,s'}^{d-1}}\right|=\frac{m_{t,s'}^d}{2bdm_{t,s}m_{t,s'}^{d-1}}\in\Z.
    \end{equation}
    When $m=d-1$, \cref{2.lem:PastInSamePiece} implies that
    $$\frac{\binom{d}{m}(\sum_{j=1}^{\ell-1}\xi_{s_j}m_{t,s_j})^{d-m}(\xi_{s'}m_{t,s'})^{m}}{2bdm_{t,s}m_{t,s'}^{d-1}}\equiv\frac{\sum_{j=1}^{\ell-1}\xi_{s_j}\xi_{s'}^mm_{t,s_j}}{2bm_{t,s}}\equiv\frac{\sum_{s_j\leq s}\xi_{s_j}\xi_{s'}^mm_{t,s_j}}{2bm_{t,s}} \mod 1.$$
    Thus, if $s\not\in\{s_1,...,s_\ell\}$ and $s>1$,
    \begin{equation}\label{2.eq:c=d,m=d-1Part1}
    \left\| \frac{\binom{d}{m}(\sum_{j=1}^{\ell-1}\xi_{s_j}m_{t,s_j})^{d-m}(\xi_{s'}m_{t,s'})^{m}}{2bdm_{t,s}m_{t,s'}^{d-1}} \right\|=\left\| \frac{\sum_{s_j< s}\xi_{s_j}\xi_{s'}^mm_{t,s_j}}{2bm_{t,s}}\right\|<\left| \frac{t^2m_{t,s-1}}{m_{t,s}}\right|<\frac{1}{t^42^{t+1}2^d},
    \end{equation}
    if $s\not\in\{s_1,...,s_\ell\}$ and $s=1$,
    \begin{equation}\label{2.eq:c=d,m=d-1part2}
    \left\| \frac{\binom{d}{m}(\sum_{j=1}^{\ell-1}\xi_{s_j}m_{t,s_j})^{d-m}(\xi_{s'}m_{t,s'})^{m}}{2bdm_{t,s}m_{t,s'}^{d-1}} \right\|=\left\| \frac{\sum_{s_j< s}\xi_{s_j}\xi_{s'}^mm_{t,s_j}}{2bm_{t,s}}\right\|=0,
    \end{equation}
    and if $s\in\{s_1,...,s_\ell\}$,
    \begin{multline}\label{2.eq:c=d,m=d-1Part3}
    \left\| \frac{\binom{d}{m}(\sum_{j=1}^{\ell-1}\xi_{s_j}m_{t,s_j})^{d-m}(\xi_{s'}m_{t,s'})^{m}}{2bdm_{t,s}m_{t,s'}^{d-1}}-\frac{\xi_s\xi_{s'}^m}{2b} \right\|\\
    \leq\left\|\frac{\xi_s\xi_{s'}^mm_{t,s}}{2bm_{t,s}} -\frac{\xi_s\xi_{s'}^m}{2b}\right\|+\left\| \frac{\sum_{s_j< s}\xi_{s_j}\xi_{s'}^mm_{t,s_j}}{2bm_{t,s}}\right\|<\frac{1}{t^42^{t+1}2^d}.
    \end{multline}
    Combining  \eqref{2.eq:c=d,m<d-1}, \eqref{2.eq:c=d,m=d}, \eqref{2.eq:c=d,m=d-1Part1}, \eqref{2.eq:c=d,m=d-1part2}, and \eqref{2.eq:c=d,m=d-1Part3}, we obtain
    $$  \left\|\frac{(\sum_{j=1}^\ell\xi_{s_j} m_{t,s_j})^d}{2bdm_{t,s}m_{t,s'}^{d-1}}-\frac{\xi_{s}\xi_{s'}^{d-1}}{2b}\right\|<\sum_{m=0}^d\binom{d}{m}\frac{1}{t^42^{t+1}2^d}=\frac{1}{t^42^{t+1}}$$
    if $s\in\{s_1,...,s_{\ell-1}\}$ and 
      $$  \left\|\frac{(\sum_{j=1}^\ell\xi_{s_j} m_{t,s_j})^d}{2bdm_{t,s}m_{t,s'}^{d-1}}\right\|<\sum_{m=0}^d\binom{d}{m}\frac{1}{t^42^{t+1}2^d}=\frac{1}{t^42^{t+1}}
      $$
      otherwise. We are done. 
\end{proof}
%%%%%%%%%%%%%%%%%%%%%%%%%%%%%
\begin{lem}\label{2.lem:DiagonalValuesPresent}
    Let $t\in\N$,  $c\in\{1,...,d\}$, $\ell\in\{1,...,t^2\}$, $\xi_{1},...,\xi_{t^2}\in\{-1,1\}$, and  $s,s_1,...,s_\ell\in\{1,...,t^2\}$ such that, when $\ell>1$, $s_1<\cdots<s_\ell$. Then,
    \begin{equation}\label{2.BoundOnDiagonalTermPresentEqual}
    \left\|\frac{(\sum_{j=1}^\ell\xi_{s_j} m_{t,s_j})^c}{2bm_{t,s}^{d}}-\frac{\xi_{s}^{d}}{2b}\right\|<\frac{1}{t^42^{t+1}}
    \end{equation}
    if $c=d$ and there exist $i\in\{1,...,\ell\}$ with $s_i=s$, and 
    \begin{equation}\label{2.BoundOnDiagonalTermPresentDifferent}
    \left\|\frac{(\sum_{j=1}^\ell\xi_{s_j} m_{t,s_j})^c}{2bm_{t,s}^{d}}\right\|<\frac{1}{t^42^{t+1}}
    \end{equation}
    otherwise.
\end{lem}
\begin{proof}
    The proof is similar to that of \cref{2.lem:OffDiagonalValuesPresent}.  If $s_\ell>s$, \cref{2.lem:PastInSamePiece} implies that 
    \begin{equation*}%\label{2.IterativeSimplificationI'}
    \frac{(\sum_{j=1}^\ell\xi_{s_j} m_{t,s_j})^c}{2bm_{t,s}^{d}}\equiv\frac{\sum_{m=0}^c\binom{c}{m}(\sum_{j=1}^{\ell-1} \xi_{s_j}m_{t,s_j})^{c-m}(\xi_{s_\ell}m_{t,s_\ell})^m}{2bm_{t,s}^d}\equiv \frac{(\sum_{j=1}^{\ell-1} \xi_{s_j}m_{t,s_j})^c}{2bm_{t,s}^{d}}\mod 1,
    \end{equation*}
    when $\ell>1$ and 
    \begin{equation*}%\label{2.IterativeSimplification2'}
    \frac{(\sum_{j=1}^\ell \xi_{s_j}m_{t,s_j})^c}{2bm_{t,s}^{d}}\equiv \frac{ (\xi_{s_{\ell}}m_{t,s_\ell})^c}{2bm_{t,s}^d}\equiv 0\mod 1
    \end{equation*}
    when $\ell=1$. Thus, in order to prove  \cref{2.lem:DiagonalValuesPresent}, we only need to show that \eqref{2.BoundOnDiagonalTermPresentEqual} and \eqref{2.BoundOnDiagonalTermPresentDifferent} hold when $s_\ell\leq s$. We will do this by considering the following cases: (i) $s_\ell<s$, (ii) $s_\ell=s$ and $\ell=1$, and (iii) $s_\ell=s$ and $\ell>1$.\\

    Suppose first that $s_\ell<s$ and note that for any $c\in\{1,...,d\}$,
    \begin{equation}\label{2.BoundForSmallSumsDiagonal'}
         \left|\frac{(\sum_{j=1}^\ell \xi_{s_j}m_{t,s_j})^c}{2bm_{t,s}}\right|<\frac{t^{2d}m_{t,s_\ell}^d}{m_{t,s}}<\frac{1}{t^42^{t+1}2^d}
    \end{equation}
    So, in particular, \eqref{2.BoundOnDiagonalTermPresentDifferent} holds when $s_\ell<s$.\\
    Next we assume that $s_\ell=s$ and $\ell=1$. When $c<d$, 
    $$\left\|\frac{(\sum_{j=1}^\ell \xi_{s_j}m_{t,s_j})^c}{2bm^d_{t,s}}\right\|=\left\|\frac{\xi_s^cm^c_{t,s}}{2bm^d_{t,s}}\right\|<\frac{1}{m_{t,s}}<\frac{1}{t^42^{t+1}2^d}$$
    and, when $c=d$,
    $$\left\|\frac{(\sum_{j=1}^\ell \xi_{s_j}m_{t,s_j})^d}{2bm^d_{t,s}}-\frac{\xi_s^d}{2b}\right\|=\left\|\frac{ \xi_{s}^dm_{t,s}^d}{2bm^d_{t,s}}-\frac{\xi_s^d}{2b}\right\|=0.$$
    Finally, we assume that $s_\ell=s$ and $\ell>1$. By \eqref{2.BoundForSmallSumsDiagonal'}, we have that for any $m\in\{0,...,d-1\}$,
    $$\left| \frac{(\sum_{j=1}^{\ell-1}\xi_{s_j}m_{t,s_j})^{d-m}(\xi_sm_{t,s})^m}{2bm_{t,s}^d}\right|<\frac{1}{t^42^{t+1}2^d}.$$
    Thus, for any $c\in\{1,...,d-1\}$,
    \begin{equation*}
        \left\| \frac{(\sum_{j=1}^\ell\xi_{s_j}m_{t,s_j})^c}{2bm_{t,s}^d}\right\|\leq 
        \sum_{m=0}^{c}\binom{c}{m}\left|\frac{(\sum_{j=1}^{\ell-1}\xi_{s_j}m_{t,s_j})^{c-m}(\xi_sm_{t,s})^m}{2bm_{t,s}^d}\right|<\sum_{m=0}^c\binom{c}{m}\frac{1}{t^42^{t+1}2^d}<\frac{1}{t^42^{t+1}}
    \end{equation*}
    and, when $c=d$,
    \begin{multline*}
    \left\| \frac{(\sum_{j=1}^\ell\xi_{s_j}m_{t,s_j})^d}{2bm_{t,s}^d}-\frac{\xi_s^d}{2b}\right\|\leq \left| \frac{\xi_s^dm_{t,s}^d}{2bm_{t,s}^d}-\frac{\xi_s^d}{2b}\right|\\
    +\sum_{m=0}^{d-1}\binom{d}{m}\left|\frac{(\sum_{j=1}^{\ell-1}\xi_{s_j}m_{t,s_j})^{d-m}(\xi_sm_{t,s})^m}{2bm_{t,s}^d}\right|<\frac{1}{t^42^{t+1}}.
    \end{multline*}
    We are done. 
\end{proof}
%%%%%%%%%%%%%%%%%%%%%%%%%%%%%%%%%%%%%%%%%%%%%%%%%%%%%%%%%%%%%%%%%%%%%%%%%%%%%%%%%%%%%%%%%%%%%%%%%%%%%%%%%%%%%%%%%%%%%
\subsection{Constructing random variables with convenient properties}
Let $p\in\Z[x]$ and $\mu$ be as in the statement of \cref{4.SpectralMeasure}. In this subsection we define the sequences of random variables $(A_N)_{N\in\N}$ and $(B_N)_{N\in\N}$ which, as we will see, will help us to compute (or, at the very least approximate)
$$\hat\mu(n)=\int_\mathbb T e^{2\pi ip(n)x}\text{d}\mu(x)$$
for each $n$ in a given IP$_0$ set $\Gamma$. As a matter of fact,  we will define $\mu$ in such a way that for each $n\in\Gamma$ there exists an $N\in\N$, for which $\hat\mu(n)$ is "well approximated" by 
$$\mathbb E(e^{i\pi(A_N+B_N+A_NB_N)}).$$
The following lemma provides useful information about $\mathbb E(e^{i\pi(A_N+B_N+A_NB_N)})$. For any $N\in\N$, we will denote the set $\{1,...,N\}$ by $[N]$.
\begin{lem}\label{2.Lem:MarkovAvarages}
Let $(X_n)_{n\in\N}$ be a sequence of independent identically distributed random variables with 
\begin{equation}
    \mathbb P(\{X_n=0\})=\mathbbm P(\{X_n=1\})=\mathbbm P(\{X_n=2\})=\frac{1}{3}.
\end{equation}
For each $N\in\N$ form the random variables
$$A_N=\left|\{n\in [N]\,|\,X_n=0\}\right|\text{ and }B_N=\left|\{n\in [N]\,|\,X_n=1\}\right|.$$
Then, for each $N\in\N$, 
\begin{equation}\label{2.eq:Average}
\mathbbm E(e^{i\pi(A_N+B_N+A_NB_N)})=p_{N,1}-p_{N,2}-p_{N,3}-p_{N,4},
\end{equation}
where
\begin{equation}\label{2.eq:ProbVector}
\begin{pmatrix} p_{N,1} \\ p_{N,2} \\ p_{N,3} \\ p_{N,4}\end{pmatrix}=
\begin{pmatrix} 
1/3 & 1/3& 1/3 & 0\\
1/3& 1/3 & 0 & 1/3\\
1/3& 0 &1/3 &1/3\\
0 &1/3 &1/3 &1/3
\end{pmatrix}^{N}
\begin{pmatrix} 1\\ 0 \\ 0 \\ 0\end{pmatrix}
\end{equation}
\end{lem}
\begin{proof}
For each $N\in\N$ let 
\begin{multline*}
    E^N_{ee}=\{A_N\in 2\Z,B_N\in 2\Z\},\,E^N_{eo}=\{A_N\in 2\Z,B_N\not\in 2\Z\},\\
    E^N_{oe}=\{A_N\not\in 2\Z,B_N\in 2\Z\},\text{ and }E^N_{oo}=\{A_N\not\in 2\Z,B_N\not\in 2\Z\}.
\end{multline*}
Observe that 
\begin{multline*}
\mathbbm E(e^{i\pi(A_N+B_N+A_NB_N)})
=\mathbbm E(\Bigg(\mathbbm 1_{E^N_{ee}}+\mathbbm 1_{E^N_{eo}}+\mathbbm 1_{E^N_{oe}}+\mathbbm 1_{E^N_{oo}}\Bigg)e^{i\pi(A_N+B_N+A_NB_N)})\\
=\mathbbm P(E^N_{ee})-\mathbbm P(E^N_{eo})-\mathbbm P(E^N_{oe})-\mathbbm P(E^N_{oo}).
\end{multline*}
Thus, in order to prove \cref{2.Lem:MarkovAvarages} it suffices to show that for each $N\in\N$, 
\begin{equation}\label{2.eq:ProbVector2}
    \mathbb P(E_{ee}^N)=p_{N,1},\,\mathbb P(E_{eo}^N)=p_{N,2},\,
    \mathbb P(E_{oe}^N)=p_{N,3},\text{ and }\mathbb P(E_{oo}^N)=p_{N,4}.
\end{equation}
We will do this by induction on $N\in\N$.\\

When $N=1$, we have $\mathbbm P(E_{ee}^1)=\mathbbm P(\{A_1=0,B_1=0\})=1/3$, $\mathbbm P(E_{eo}^1)=\mathbbm P(\{B_1=1\})=1/3$, $\mathbbm P(E_{oe}^1)=\mathbbm P(\{A_1=1\})=1/3$, and $\mathbbm P(E_{oo}^1)=\mathbbm P(\{A_1=1,B_1=1\})=0$. So, \eqref{2.eq:ProbVector2} holds when $N=1$.\\
Fix now $N\in\N$ and suppose that \eqref{2.eq:ProbVector2} holds for this particular value of $N$. Note that 
\begin{multline}
\mathbbm P(E^{N+1}_{ee})\\
=\mathbb P(E^{N+1}_{ee}\cap \{X_{N+1}=2\})+\mathbb P(E^{N+1}_{ee}\cap \{X_{N+1}=1\})+\mathbb P(E^{N+1}_{ee}\cap \{X_{N+1}=0\})\\
=\mathbb P(E^{N}_{ee}\cap \{X_{N+1}=2\})+\mathbb P(E^{N}_{eo}\cap \{X_{N+1}=1\})+\mathbb P(E^{N}_{oe}\cap \{X_{N+1}=0\})\\
=\frac{1}{3}\Bigg(\mathbbm P(E^N_{ee})+\mathbbm P(E^N_{eo})+\mathbbm P(E^N_{oe})\Bigg).
\end{multline}
In a similar way one can show that
\begin{align*}
    \mathbbm P(E^{N+1}_{eo})&=\frac{1}{3}\Bigg(\mathbbm P(E^N_{ee})+\mathbbm P(E^N_{eo})+\mathbbm P(E^N_{oo})\Bigg)\\
    \mathbbm P(E^{N+1}_{oe})&=\frac{1}{3}\Bigg(\mathbbm P(E^N_{ee})+\mathbbm P(E^N_{oe})+\mathbbm P(E^N_{oo})\Bigg)\\
    \mathbbm P(E^{N+1}_{oo})&=\frac{1}{3}\Bigg(\mathbbm P(E^N_{eo})+\mathbbm P(E^N_{oe})+\mathbbm P(E^N_{oo})\Bigg).
\end{align*}
Thus,
\begin{equation*}
\begin{pmatrix}\mathbbm P(E^{N+1}_{ee}) \\ \mathbbm P(E^{N+1}_{eo})\\ \mathbbm P(E^{N+1}_{oe})\\ \mathbbm P(E^{N+1}_{oo})\end{pmatrix}=
\begin{pmatrix} 
1/3 & 1/3& 1/3 & 0\\
1/3& 1/3 & 0 & 1/3\\
1/3& 0 &1/3 &1/3\\
0 &1/3 &1/3 &1/3
\end{pmatrix}
\begin{pmatrix} \mathbbm P(E^{N}_{ee}) \\ \mathbbm P(E^{N}_{eo}) \\ \mathbbm P(E^{N}_{oe}) \\ \mathbbm P(E^{N}_{oo})\end{pmatrix},
\end{equation*}
completing the induction. We are done. 
\end{proof}
%%%%%%%%%%%%%%%%%%%%%%%%%%%%%%%%%%%%%%%%%%%%%%%%%%%%%%%%%%%%%%%%%%%%%%%%%%%%%%%%%%%%%%%%%%%%%%%%%%%%%%%%%%%%
\subsection{Proof of \cref{4.SpectralMeasure}}
The following argument parallels the one used to prove \cite[Theorem 3.1]{zelada2022GaussianETDS}.
\begin{proof}[Proof of \cref{4.SpectralMeasure}]
Let $d=\deg(p)$, so $d>1$, let $b_1,...,b_d\in\Z$ be such that 
$$p(x)=b_dx^d+\cdots+b_1x$$
(so, in particular, $b_d\neq 0$), and let $((m_{t,s})_{s=1}^{t^2})_{t\in\N}$ be a family of finite sequences in $\N$ for which conditions (C.1)-(C.5) in Subsection 4.1 hold with $b=|b_d|$.
To prove \cref{4.SpectralMeasure}, we will first define the probability measure $\mu$ and then prove that \eqref{3.eq:NegativeCorrelations} holds.

\textit{\qedsymbol Defining $\mu$}:  Let $\mathcal C=\prod_{t=1}^\infty\prod_{s=1}^{t^2}\{0,1,2\}$ be endowed with the product topology and let $\mathbb P$ be the product measure on $\mathcal C$ defined by the property  that for any $t\in\N$ and any $s\in\{1,...,t^2\}$,
$$\mathbb P(\{\omega\in\mathcal C\,|\,\omega(t,s)=0\})=\mathbb P(\{\omega\in\mathcal C\,|\,\omega(t,s)=1\})=\mathbb P(\{\omega\in\mathcal C\,|\,\omega(t,s)=2\})=\frac{1}{3}.$$
 Define the map $g:\mathcal C\rightarrow \R$ by 
$$g(\omega)=\sum_{t=2}^\infty\left(\sum_{s=1}^{t^2}\frac{\mathbbm 1_{\{0,1\}}(\omega(t,s))}{2bm_{t,s}^d}+\sum_{s'=2}^{t^2}\sum_{s=1}^{s'-1}\frac{\mathbbm 1_{\{0\}}(\omega(t,s))\mathbbm 1_{\{1\}}(\omega(t,s'))+\mathbbm 1_{\{1\}}(\omega(t,s))\mathbbm 1_{\{0\}}(\omega(t,s'))}{2bdm_{t,s}m_{t,s'}^{d-1}} \right)$$
Observe that for any $\omega\in\mathcal C$ and any $t>1$, 
\begin{multline*}
    \left|\sum_{s=1}^{t^2}\frac{\mathbbm 1_{\{0,1\}}(\omega(t,s))}{2bm_{t,s}^d}+\sum_{s'=2}^{t^2}\sum_{s=1}^{s'-1}\frac{\mathbbm 1_{\{0\}}(\omega(t,s))\mathbbm 1_{\{1\}}(\omega(t,s'))+\mathbbm 1_{\{1\}}(\omega(t,s))\mathbbm 1_{\{0\}}(\omega(t,s'))}{2bdm_{t,s}m_{t,s'}^{d-1}}\right|\\
    \leq \sum_{s=1}^{t^2}\frac{1}{2m_{t,s}^d}+\sum_{s'=2}^{t^2}\sum_{s=1}^{s'-1}\frac{1}{2dm_{t,s}m_{t,s'}^{d-1}}.
\end{multline*}
Thus, by Weierstrass M-test and \cref{2.lem:BoundOftheFuture}, $g$ is continuous. 
Let $\phi$ be the canonical map from $\R$ to $[0,1)= \R/\Z$ (so $\phi(x) = x \mod 1$ and $\phi$ is continuous) and let $f=\phi\circ g$.
Since $\phi\circ g$ is continuous, we have that $f$ is measurable. Let
\begin{equation}\label{4.eq:DefnMu}
\mu=\mathbb P\circ f^{-1}.
\end{equation}
Clearly, $\mu$ is a Borel probability measure defined on $\mathbb T$.\\

\textit{\qedsymbol Preliminary estimates for the proof of \eqref{3.eq:NegativeCorrelations}}: Fix now $\tau\in\N$, $\tau>2$. Let $F\subseteq\{1,...,\tau^2\}$ be such that $|F|\geq \tau$ and let $\xi_{1},...,\xi_{\tau^2}\in\{-1,1\}$. For each $\omega\in \mathcal C$ define the functions 
$$A_{\tau,F}(\omega)=|\{j\in F\,|\,\omega(\tau,j)=0\}|$$
and 
$$B_{\tau,F}(\omega)=|\{j\in F\,|\,\omega(\tau,j)=1\}|.$$
Set $n=\sum_{j\in F} \xi_{j}m_{\tau,j}$. We claim that 
\begin{equation}\label{3.eq:MainApproximation}
\left\|p(n)f(\omega)-\frac{A_{\tau,F}(\omega)+B_{\tau,F}(\omega)+A_{\tau,F}(\omega)B_{\tau,F}(\omega)}{2}\right\|<\frac{\sum_{c=1}^d|b_c|}{2^{\tau-1}}
\end{equation}
for any $\omega\in\mathcal C$. 
Indeed, by \cref{2.lem:PastPiecesAreIntegers}, we have that 
\begin{multline*}
p(n)\sum_{t=2}^{\tau-1}\left(\sum_{s=1}^{t^2}\frac{\mathbbm 1_{\{0,1\}}(\omega(t,s))}{2bm_{t,s}^d}\right.\\
\left.+\sum_{s'=2}^{t^2}\sum_{s=1}^{s'-1}\frac{\mathbbm 1_{\{0\}}(\omega(t,s))\mathbbm 1_{\{1\}}(\omega(t,s'))+\mathbbm 1_{\{1\}}(\omega(t,s))\mathbbm 1_{\{0\}}(\omega(t,s'))}{2bdm_{t,s}m_{t,s'}^{d-1}} \right)\in\Z
\end{multline*}
 by \cref{2.lem:BoundOftheFuture} that,
\begin{multline*}
\left|p(n)\sum_{t=\tau+1}^{\infty}\left(\sum_{s=1}^{t^2}\frac{\mathbbm 1_{\{0,1\}}(\omega(t,s))}{2bm_{t,s}^d}\right.\right.\\
\left.\left.+\sum_{s'=2}^{t^2}\sum_{s=1}^{s'-1}\frac{\mathbbm 1_{\{0\}}(\omega(t,s))\mathbbm 1_{\{1\}}(\omega(t,s'))+\mathbbm 1_{\{1\}}(\omega(t,s))\mathbbm 1_{\{0\}}(\omega(t,s'))}{2bdm_{t,s}m_{t,s'}^{d-1}} \right)\right|\\
\leq |p(n)|\sum_{t=\tau+1}^\infty\left(\sum_{s=1}^{t^2}\frac{1}{2m_{t,s}^d}+\sum_{s'=2}^{t^2}\sum_{s=1}^{s'-1}\frac{1}{2dm_{t,s}m_{t,s'}^{d-1}}\right)<\frac{|p(n)|}{\tau^{2d}m_{\tau,\tau^2}^d2^\tau}\\
\leq \frac{\sum_{c=1}^d|b_c|(\tau^2 m_{\tau,\tau^2})^d}{\tau^{2d}m_{\tau,\tau^2}^d2^\tau}<\frac{\sum_{c=1}^d|b_c|}{2^\tau}.
\end{multline*}
and, by Lemmas \ref{2.lem:OffDiagonalValuesPresent} and \ref{2.lem:DiagonalValuesPresent},
\begin{multline*} 
\left\|(\sum_{c=1}^{d-1}b_cn^c)\left(\sum_{s=1}^{\tau^2}\frac{\mathbbm 1_{\{0,1\}}(\omega(\tau,s))}{2bm_{\tau,s}^d}+\sum_{s'=2}^{\tau^2}\sum_{s=1}^{s'-1}\frac{\mathbbm 1_{\{0\}}(\omega(\tau,s))\mathbbm 1_{\{1\}}(\omega(\tau,s'))+\mathbbm 1_{\{1\}}(\omega(\tau,s))\mathbbm 1_{\{0\}}(\omega(\tau,s'))}{2bdm_{\tau,s}m_{\tau,s'}^{d-1}}\right)\right\|\\
\leq\sum_{c=1}^{d-1}|b_c|\left\|n^c\sum_{s=1}^{\tau^2}\frac{\mathbbm 1_{\{0,1\}}(\omega(\tau,s))}{2bm_{\tau,s}^d}+n^c\sum_{s'=2}^{\tau^2}\sum_{s=1}^{s'-1}\frac{\mathbbm 1_{\{0\}}(\omega(\tau,s))\mathbbm 1_{\{1\}}(\omega(\tau,s'))+\mathbbm 1_{\{1\}}(\omega(\tau,s))\mathbbm 1_{\{0\}}(\omega(\tau,s'))}{2bdm_{\tau,s}m_{\tau,s'}^{d-1}} \right\|\\
<\frac{\tau^4\sum_{c=1}^{d-1}|b_c|}{\tau^42^{\tau+1}}=\frac{\sum_{c=1}^{d-1}|b_c|}{2^{\tau+1}}.
\end{multline*}
Thus,
\begin{multline}\label{3.eq:InequalityAtNthTime1}
    \left\|p(n)f(\omega)-\frac{A_{\tau,F}(\omega)+B_{\tau,F}(\omega)+A_{\tau,F}(\omega)B_{\tau,F}(\omega)}{2}\right\|\\<\left\|b_dn^d\sum_{s'=2}^{\tau^2}\sum_{s=1}^{s'-1}\frac{\mathbbm 1_{\{0\}}(\omega(\tau,s))\mathbbm 1_{\{1\}}(\omega(\tau,s'))+\mathbbm 1_{\{1\}}(\omega(\tau,s))\mathbbm 1_{\{0\}}(\omega(\tau,s'))}{2bdm_{\tau,s}m_    {\tau,s'}^{d-1}}-\frac{A_{\tau,F}(\omega)B_{\tau,F}(\omega)}{2}
    \right\|\\
    +\left\|b_dn^d\sum_{s=1}^{\tau^2}\frac{\mathbbm 1_{\{0,1\}}(\omega(\tau,s))}{2bm_{\tau,s}^d}-\frac{A_{\tau,F}(\omega)+B_{\tau,F}(\omega)}{2}\right\|+\frac{\sum_{c=1}^{d-1}|b_c|}{2^{\tau+1}}+\frac{\sum_{c=1}^{d}|b_c|}{2^{\tau}}.
\end{multline}
Since $-\frac{1}{2}\equiv \frac{1}{2} \mod 1$ (and, hence, $-\frac{k}{2}\equiv\frac{k}{2}\mod 1$ for each $k\in\Z$) and $((m_{t,s})_{s=1}^{t^2})_{t\in\N}$ satisfies conditions (C.1)-(C.5) when $b=1$, \cref{2.lem:OffDiagonalValuesPresent} implies that 
\begin{multline}\label{3.eq:InequalityAtNthTime2}
    \left\|b_dn^d\sum_{s'=2}^{\tau^2}\sum_{s=1}^{s'-1}\frac{\mathbbm 1_{\{0\}}(\omega(\tau,s))\mathbbm 1_{\{1\}}(\omega(\tau,s'))+\mathbbm 1_{\{1\}}(\omega(\tau,s))\mathbbm 1_{\{0\}}(\omega(\tau,s'))}{2bdm_{\tau,s}m_    {\tau,s'}^{d-1}}-\frac{A_{\tau,F}(\omega)B_{\tau,F}(\omega)}{2}
    \right\|\\
    %%%%%%%%%%%%%%%%%%%%%%%%%%%%%%%%
=\left\|\frac{b_d}{|b_d|}\left(n^d\sum_{s'=2}^{\tau^2}\sum_{s=1}^{s'-1}\frac{\mathbbm 1_{\{0\}}(\omega(\tau,s))\mathbbm 1_{\{1\}}(\omega(\tau,s'))+\mathbbm 1_{\{1\}}(\omega(\tau,s))\mathbbm 1_{\{0\}}(\omega(\tau,s'))}{2dm_{\tau,s}m_    {\tau,s'}^{d-1}}-\frac{A_{\tau,F}(\omega)B_{\tau,F}(\omega)}{2}
    \right)\right\|\\
    %%%%%%%%%%%%%%%%%%%%%%%%%%%%%%%%%%%%
=\left\|n^d\sum_{s'=2}^{\tau^2}\sum_{s=1}^{s'-1}\frac{\mathbbm 1_{\{0\}}(\omega(\tau,s))\mathbbm 1_{\{1\}}(\omega(\tau,s'))+\mathbbm 1_{\{1\}}(\omega(\tau,s))\mathbbm 1_{\{0\}}(\omega(\tau,s'))}{2dm_{\tau,s}m_    {\tau,s'}^{d-1}}-\frac{A_{\tau,F}(\omega)B_{\tau,F}(\omega)}{2}\right\|\\
%%%%%%%%%%%%%%%
=\left\|n^d\sum_{s\in\{j\in [\tau^2]\,|\,\omega(\tau,j)=0\}}\sum_{s'\in\{j\in [\tau^2]\,|\,\omega(\tau,j)=1\}}\frac{1}{2dm_{\tau,\min\{s,s'\}}m_    {\tau,\max\{s,s'\}}^{d-1}}-\frac{A_{\tau,F}(\omega)B_{\tau,F}(\omega)}{2}
    \right\|\\
%%%%%%%%%%%%%%%%%%%%%%%%%%%%%%%%%%%%%%%%%%%%%%%%%%%%
\leq  \left\|n^d\sum_{s\in\{j\in [\tau^2]\,|\,\omega(\tau,j)=0\}}\sum_{s'\in\{j\in [\tau^2]\,|\,\omega(\tau,j)=1\}}\frac{1}{2dm_{\tau,\min\{s,s'\}}m_    {\tau,\max\{s,s'\}}^{d-1}}\right.\\
%%%%%%%%%%%%%%5
\left.-\sum_{s\in\{j\in F\,|\,\omega(\tau,j)=0\}}\sum_{s'\in\{j\in F\,|\,\omega(\tau,j)=1\}}\frac{1}{2}\right\|
+\left\|\sum_{s\in\{j\in F\,|\,\omega(\tau,j)=0\}}\sum_{s'\in\{j\in F\,|\,\omega(\tau,j)=1\}}\frac{1}{2}-\frac{A_{\tau,F}(\omega)B_{\tau,F}(\omega)}{2}\right\|\\
    %%%%%%%%%%%%%%%%%%%%%%%%%%
    <\frac{\tau^4-\tau^2}{\tau^42^{\tau+1}}+\left\|\sum_{s\in\{j\in F\,|\,\omega(\tau,j)=0\}}\sum_{s'\in\{j\in F\,|\,\omega(\tau,j)=1\}}\frac{1}{2}-\frac{A_{\tau,F}(\omega)B_{\tau,F}(\omega)}{2}\right\|=\frac{\tau^4-\tau^2}{\tau^42^{\tau+1}}
\end{multline}
and \cref{2.lem:DiagonalValuesPresent} implies 
\begin{multline}\label{3.eq:InequalityAtNthTime3}
    \left\|b_dn^d\sum_{s=1}^{\tau^2}\frac{\mathbbm 1_{\{0,1\}}(\omega(\tau,s))}{2bm_{\tau,s}^d}-\frac{A_{\tau,F}(\omega)+B_{\tau,F}(\omega)}{2}\right\|\\
    \leq \left\|\sum_{s\in \{j\in [\tau^2]\,|\,\omega(\tau,j)\in\{0,1\}\}}\frac{n^d}{2m_{\tau,s}^d}
    -\sum_{s\in\{j\in F\,|\,\omega(\tau,j)\in\{0,1\}\}}\frac{1}{2}\right\|\\
    +\left\|\sum_{s\in\{j\in F\,|\,\omega(\tau,j)\in\{0,1\}\}}\frac{1}{2}-\frac{A_{\tau,F}(\omega)+B_{\tau,F}(\omega)}{2}\right\|\\
    <\frac{\tau^2}{\tau^42^{\tau+1}}+\left\|\sum_{s\in\{j\in F\,|\,\omega(\tau,j)\in\{0,1\}\}}\frac{1}{2}-\frac{A_{\tau,F}(\omega)+B_{\tau,F}(\omega)}{2}\right\|=\frac{\tau^2}{\tau^42^{\tau+1}}.
\end{multline}
Combining \eqref{3.eq:InequalityAtNthTime1}, \eqref{3.eq:InequalityAtNthTime2}, and \eqref{3.eq:InequalityAtNthTime3}, we see that \eqref{3.eq:MainApproximation} holds.\\

\textit{\qedsymbol Proof of \eqref{3.eq:NegativeCorrelations}}: It now follows from \eqref{3.eq:MainApproximation} that
\begingroup 
\allowdisplaybreaks
\begin{multline*}
\limsup_{t\rightarrow\infty}\left(\max_{F\subseteq \{1,...,t^2\},\,|F|\geq t}\max_{\xi_1,...,\xi_{t^2}\in\{-1,1\}}\right.\\
\left.\left|\int_\mathbb T e^{2\pi ip(\sum_{j\in F}\xi_jm_{t,j})x}\text{d}\mu(x)-\int_{\mathcal C}e^{\pi i(A_{t,F}(\omega)+B_{t,F}(\omega)+A_{t,F}(\omega)B_{t,F}(\omega))}\text{d}\mathbb P(\omega)\right|\right)\\
=\limsup_{t\rightarrow\infty}\left(\max_{F\subseteq \{1,...,t^2\},\,|F|\geq t}\max_{\xi_1,...,\xi_{t^2}\in\{-1,1\}}\right.\\
\left.\left|\int_\mathcal C e^{2\pi ip(\sum_{j\in F}\xi_jm_{t,j})f(\omega)}\text{d}\mathbb P(\omega)-\int_{\mathcal C}e^{\pi i(A_{t,F}(\omega)+B_{t,F}(\omega)+A_{t,F}(\omega)B_{t,F}(\omega))}\text{d}\mathbb P(\omega)\right|\right)=0.
\end{multline*}
\endgroup
Thus, in order to prove \eqref{3.eq:NegativeCorrelations}, all we need to show is that 
\begin{equation}\label{3.eq:Main'}
\lim_{t\rightarrow\infty}\left(\max_{F\subseteq \{1,...,t^2\},\,|F|\geq t}\max_{\xi_1,...,\xi_{t^2}\in\{-1,1\}}\left|\int_{\mathcal C}e^{\pi i(A_{t,F}(\omega)+B_{t,F}(\omega)+A_{t,F}(\omega)B_{t,F}(\omega))}\text{d}\mathbb P(\omega)+\frac{1}{2}\right|\right)=0.
\end{equation}
To see this, let  $t>2$ and let $F\subseteq \{1,...,t^2\}$ be such that $|F|\geq t$. Observe that $A_{t,F}$ and $B_{t,F}$ only depend on the independent and identically distributed random variables $(X_{t,j})_{j\in F}$, where for each $\omega\in \mathcal C$, $X_{t,j}(\omega)=\omega(t,j)$. Thus, by \cref{2.Lem:MarkovAvarages},
\begin{multline*}
\int_{\mathcal C}e^{\pi i(A_{t,F}(\omega)+B_{t,F}(\omega)+A_{t,F}(\omega)B_{t,F}(\omega))}\text{d}\mathbb P(\omega)=\mathbb E(e^{i\pi(A_{t,F}+B_{t,F}+A_{t,F}B_{t,F})})\\
=\mathbb E(e^{i\pi(A_{|F|}+B_{|F|}+A_{|F|}B_{|F|})})=p_{|F|,1}-p_{|F|,2}-p_{|F|,3}-p_{|F|,4},
\end{multline*}
where $p_{|F|,j}$, $j=1,...,4$ are as defined in \eqref{2.eq:ProbVector}. 
Noting  that
\begin{equation*}
\lim_{N\rightarrow\infty}
\begin{pmatrix} 
1/3 & 1/3& 1/3 & 0\\
1/3& 1/3 & 0 & 1/3\\
1/3& 0 &1/3 &1/3\\
0 &1/3 &1/3 &1/3
\end{pmatrix}^{N}
\begin{pmatrix} 1 \\ 0 \\ 0 \\ 0\end{pmatrix}=\begin{pmatrix} 1/4 \\ 1/4 \\ 1/4 \\ 1/4\end{pmatrix},
\end{equation*}
we see that \eqref{3.eq:Main'} holds. We are done. 
\end{proof}
%%%%%%%%%%%%%%%%%%%%%%%%%%%%%%%%%%%%%%%%%%%%%%%%%%%%%%%%%%%%%%%%%%%%%%%%%%%%%%%%%%%%%%%%
%%%%%%%%%%%%%%%%%%%%%%%%%%%%%%%%%%%%%%%%%%%%%%%%%%%%%%%%%%%%%%%%%%%%%%%%%%%%%%%%%%%%%%%%%%5
\section{Proof of Proposition B}
In this section we prove the following result from which Proposition B follows.
\begin{namedthm*}{Proposition B$^\prime$}
The Gaussian system $(\R^\Z,\mathcal A,\gamma, T)$ in Theorem A$^\prime$ (and, hence, the
invertible probability preserving system $(X,\mathcal A,\mu, T)$ in the statement of Theorem A) can be picked to be both   weakly mixing  and rigid. 
\end{namedthm*}
\begin{proof}[Proof of Proposition B$^\prime$]
    Let $(\R^\Z,\mathcal A,\gamma, T)$ be the Gaussian system utilized in the proof of Theorem A$^\prime$ (i.e. $(\R^\Z,\mathcal A,\gamma,T)$ is the Gaussian system associated with $\rho$) and let $((m_{t,s})_{s=1}^{t^2})_{t\in\N}$ be the sequence guaranteed to exist in \cref{4.SpectralMeasure}. For each $t\in\N$, we will let $n_t=\sum_{j=1}^{t^2}m_{t,j}$.\\
    We will first show that $(\R^\Z,\mathcal A,\gamma,T)$ is rigid.
    By \eqref{3.eq:MainApproximation}     $$\lim_{t\rightarrow\infty}\int_\mathcal Ce^{4\pi ip(n_t)f(\omega)}\text{d}\mathbb P(\omega)=1$$
   and, hence, by the definition of $\rho$ (and $\mu$), 
$$\lim_{t\rightarrow\infty}\int_\mathbb Te^{4\pi ip(n_t)x}\text{d}\rho(x)=1.$$
It follows that    $$\lim_{t\rightarrow\infty}e^{4\pi ip(n_t)x}=1$$
in $\rho$-measure and, since $|e^{4\pi ip(n_t)x}|\leq 1$ for each $x\in \mathbb T$,
$$\lim_{t\rightarrow\infty}\int_\mathbb T|e^{4\pi ip(n_t)x}-1|\text{d}\rho(x)=0.$$
 So, for any $m\in\Z$,
 $$\lim_{t\rightarrow\infty}\int_\mathbb T e^{2\pi imx} e^{2\pi i(2p(n_t))x}\text{d}\rho(x)=\int_\mathbb T e^{2\pi i mx}\text{d}\rho(x)$$
    which, by (G.5) above, implies that $T$ is rigid.\\

    By (G.4), in order to prove that $(\R^\Z,\mathcal A,\gamma, T)$ is weakly mixing, it suffices to show that there exists a sequence $(m_k)_{k\in\N}$ such that for any $m\in\Z$,
    $$\lim_{k\rightarrow\infty}\int_\mathbb Te^{2\pi im_kx}e^{2\pi imx}\text{d}\rho(x)=0.$$
    In turn, by the definition of $\rho$, it is enough to prove that  there exists a sequence $(m_k)_{k\in\N}$ such that for any $m\in\Z$,
    $$\lim_{k\rightarrow\infty}\int_\mathbb Te^{2\pi im_kx}e^{2\pi imx}\text{d}\mu(x)=0.$$
    For this, note that by \eqref{4.eq:DefnMu}, \eqref{3.eq:MainApproximation}, and \eqref{3.eq:Main'}, for any $m\in\Z$,
   \begin{multline*} \lim_{t\rightarrow\infty}\int_\mathbb T e^{2\pi i p(n_{t})x}e^{2\pi imx}\text{d}\mu(x)
=\lim_{t\rightarrow\infty}\int_\mathcal C e^{\pi i(A_{t,[t^2]}(\omega)+B_{t,[t^2]}(\omega)+A_{t,[t^2]}(\omega)B_{t,[t^2]}(\omega))}e^{2\pi imf(\omega)}\text{d}\mathbb P(\omega)\\
=\lim_{t\rightarrow\infty}\int_\mathcal C e^{\pi i(A_{t,[t^2]}(\omega)+B_{t,[t^2]}(\omega)+A_{t,[t^2]}(\omega)B_{t,[t^2]}(\omega))}\text{d}\mathbb P(\omega)\int_\mathcal C e^{2\pi imf(\omega)}\text{d}\mathbb P(\omega)=-\frac{1}{2}\int_\mathbb Te^{2\pi imx}\text{d}\mu(x)
\end{multline*}
and, hence, for  any $N\in\N$, 
$$\lim_{t_1\rightarrow\infty}\cdots\lim_{t_N\rightarrow\infty}\int_\mathbb T e^{2\pi i\sum_{j=1}^N p(n_{t_j})x}e^{2\pi imx}\text{d}\mu(x)=\left(-\frac{1}{2}\right)^N\int_\mathbb Te^{2\pi imx}\text{d}\mu(x).$$
Thus, for each $k\in\N$ we can find natural numbers $t^{(k)}_1<\cdots<t^{(k)}_k$ for which the sequence $m_k=\sum_{j=1}^kp(n_{t^{(k)}_j})$, $k\in\N$, satisfies
$$\lim_{k\rightarrow\infty}\int_{\mathbb T} e^{2\pi im_kx}e^{2\pi imx}\text{d}\mu(x)=\lim_{k\rightarrow\infty}(-\frac{1}{2})^k\int_\mathbb Te^{2\pi imx}\text{d}\mu(x)=0$$
for any $m\in\Z$. We are done. 
\end{proof}
 \bibliography{Bib.bib}

\providecommand{\bysame}{\leavevmode\hbox to3em{\hrulefill}\thinspace}
\providecommand{\MR}{\relax\ifhmode\unskip\space\fi MR }
% \MRhref is called by the amsart/book/proc definition of \MR.
\providecommand{\MRhref}[2]{%
  \href{http://www.ams.org/mathscinet-getitem?mr=#1}{#2}
}
\providecommand{\href}[2]{#2}
\begin{thebibliography}{10}

\bibitem{BerUltraAcrossMath}
V.~Bergelson, \emph{Ultrafilters, {I}{P} sets, dynamics, and combinatorial
  number theory}, Ultrafilters across {M}athematics, Contemporary
  {M}athematics, vol. 530, American Mathematical Society, Providence, RI, 2010,
  pp.~23--47.

\bibitem{BFM}
V.~Bergelson, H.~Furstenberg, and R.~McCutcheon, \emph{{IP}\text{-}sets and
  polynomial recurrence}, Ergodic Theory and Dynamical Systems \textbf{16}
  (1996), no.~5, 963--974.

\bibitem{IPrRecNilsystems}
V.~Bergelson and A.~Leibman, \emph{I{P}$_r^*\text-$recurrence and nilsystems},
  Advances in Mathematics \textbf{339} (2018), 642--656.

\bibitem{AlmostIPBerLeib}
\bysame, \emph{Sets of large values of correlation functions for polynomial
  cubic configurations}, Ergodic Theory and Dynamical Systems \textbf{38}
  (2018), no.~2, 499--522.

\bibitem{bergelsonShiftedPrimes}
V.~Bergelson, A.~Leibman, and T.~Ziegler, \emph{The shifted primes and the
  multidimensional {S}zemer{\'e}di and polynomial van der {W}aerden theorems},
  Comptes Rendus Math{\'e}matique \textbf{349} (2011), no.~3-4, 123--125.

\bibitem{berMcCuIPPolySzemeredi}
V.~Bergelson and R.~McCutcheon, \emph{An ergodic {I}{P} polynomial
  {S}zemer{\'e}di theorem}, Memoirs of the American Mathematical Society
  \textbf{146} (2000), no.~695, 106 pp.

\bibitem{BDonaldRobertsonIP_r}
V.~Bergelson and D.~Robertson, \emph{Polynomial recurrence with large
  intersection over countable fields}, Israel Journal of Mathematics
  \textbf{214} (2016), no.~1, 109--120.

\bibitem{BerZel_JCTA_iteratedDifferences2021}
V.~Bergelson and R.~Zelada, \emph{Iterated differences sets, {D}iophantine
  approximations and applications}, Journal of Combinatorial Theory, Series A
  \textbf{184} (2021), Paper No. 105520, 51pp.

\bibitem{cornfeld1982ergodic}
I.~P. Cornfeld, S.~V. Fomin, and Ya.~G. Sinai, \emph{Ergodic theory},
  {G}rundlehren der {M}athematischen {W}issenschaften, vol. 245, Springer,
  1982.

\bibitem{FBook}
H.~Furstenberg, \emph{Recurrence in ergodic theory and combinatorial number
  theory}, Princeton University Press, 1981.

\bibitem{FKIPSzemerediLong}
H.~Furstenberg and Y.~Katznelson, \emph{An ergodic {S}zemer{\'e}di theorem for
  {I}{P}-systems and combinatorial theory}, Journal d'Analyse Math{\'e}matique
  \textbf{45} (1985), 117--168.

\bibitem{HIPPartitionRegular}
N.~Hindman, \emph{Finite sums from sequences within cells of a partition of
  $\mathbb{N}$}, Journal of Combinatorial Theory, Series A \textbf{17} (1974),
  1--11.

\bibitem{kechris2010Global}
A.~S. Kechris, \emph{Global aspects of ergodic group actions}, American
  Mathematical Society, 2010.

\bibitem{zelada2022GaussianETDS}
R.~Zelada, \emph{Mixing and rigidity along asymptotically linearly independent
  sequences}, Ergodic Theory and Dynamical Systems (2022), 1--32.

\end{thebibliography}
\bibliographystyle{amsplain}
\end{document}